\documentclass[11pt,letterpaper,reqno]{amsart}

\usepackage[T1]{fontenc}
\usepackage{amsmath,amssymb,amsthm,amsfonts}
\usepackage{enumitem}
\usepackage{graphicx}
\usepackage{booktabs}

\usepackage{hyperref}
\hypersetup{
	pdfstartview={FitH},
	colorlinks=false,
	pdfborder={0 0 0},
	hidelinks
}

\usepackage{bookmark}

\usepackage{doi}

\newtheorem{thm}{Theorem}[section]
\newtheorem{lem}[thm]{Lemma}
\newtheorem{prop}[thm]{Proposition}
\newtheorem{cor}[thm]{Corollary}

\newtheorem{conj}[thm]{Conjecture}

\theoremstyle{definition}

\newtheorem{remark}[thm]{Remark}
\newtheorem{defn}[thm]{Definition}

\numberwithin{equation}{section}

\newcommand{\sn}{\mathbb{S}_n}          % real symmetric matrices (typographically distinct from \Sym)
\newcommand{\R}{\mathbb{R}}
\newcommand{\one}{\mathbf{1}}  % all-ones vector
\newcommand{\Sym}{\mathfrak{S}_n}       % symmetric group

\DeclareMathOperator{\tr}{tr}
\DeclareMathOperator{\diag}{diag}

\newcommand{\maj}{\prec}

\newcommand{\Rey}{\mathcal{R}}      % Reynolds operator (generic)
\begin{document}
	
	\title[Schur--Horn type inequalities for hyperbolic polynomials]
	{Schur--Horn type inequalities for hyperbolic polynomials}
	
	\author[T.~Zhang]{Teng Zhang}
	\address{School of Mathematics and Statistics, Xi'an Jiaotong University, Xi'an 710049, P. R. China}
	\email{teng.zhang@stu.xjtu.edu.cn}
	
	
	\begin{abstract}
We establish a Schur--Horn type inequality for symmetric hyperbolic polynomials. As an immediate consequence, we resolve a conjecture of Nam Q. Le on Hadamard-type inequalities for hyperbolic polynomials. Our argument is based on Schur–-Horn theorem, Birkhorff theorem and Gårding’s concavity theorem for hyperbolicity cones.
Beyond the eigenvalue level, we develop a symmetrization principle on hyperbolicity cones:
if a hyperbolic polynomial is invariant under a finite group action, then its value increases under the associated Reynolds operator (group averaging).
Applied to the sign-flip symmetries of linear principal minor polynomials introduced by Blekherman et al.,
this yields a short  proof of the hyperbolic Fischer--Hadamard inequalities for PSD-stable lpm polynomials.
	\end{abstract}
	
	\maketitle
	
	% ============================================================
	\section{Introduction}\label{sec:intro}
	
	Let $\sn$ denote the set of all $n\times n$ real symmetric matrices. For $A=(a_{ij})\in\sn$, we write
	\[
	\lambda(A)=(\lambda_1(A),\ldots,\lambda_n(A))
	\]
	for the list of eigenvalues of $A$ and
	\[
	\diag(A)=(a_{11},\ldots,a_{nn})
	\]
	for its vector of diagonal entries.
	We also let $\one=(1,\ldots,1)\in\R^n$.
	
	For $\lambda=(\lambda_1,\dots,\lambda_n)\in\R^n$ and $1\le k\le n$, we denote by
	\[
	S_k(\lambda)
	:= \!\!\sum_{1\le i_1<\cdots<i_k\le n}\!\!\lambda_{i_1}\cdots\lambda_{i_k}
	\]
	the $k$th elementary symmetric function, and adopt the convention $S_0(\lambda)=1$.
	When $A\in \sn$, we write $S_k(A):=S_k(\lambda(A))$.
	
	\subsection{G{\aa}rding cones and $k$-positive matrices}
	
	We define the open, symmetric, convex cone
	\[
	\Gamma_k(n)\ :=\ \bigl\{\lambda\in\R^n:\ S_j(\lambda)>0\ \text{for all }1\le j\le k\bigr\}.
	\]
	The convexity of $\Gamma_k(n)$ is a standard consequence of G{\aa}rding's theory of
	hyperbolic polynomials; see \cite[Example~5.2]{Le22}.
	
	\begin{defn}[$k$-positive matrix {\cite[Definition~1.1]{Le22}}]
		A matrix $A\in \sn$ is called \emph{$k$-positive} if $\lambda(A)\in\Gamma_k(n)$.
	\end{defn}
	
	The classical Hadamard determinant inequality
	\cite[p.~505, Theorem~7.8.1]{HJ13} can be phrased as follows.
	
	\begin{thm}[Hadamard]\label{thm:hadamard}
		Let $A=(a_{ij})\in \sn$ be $n$-positive. Then
		\[
		S_n(A)\ \le\ S_n(\diag(A)).
		\]
		Moreover, equality holds if and only if $A$ is diagonal.
	\end{thm}
	
	Hadamard's inequality has been generalized and refined in many directions;
	see for instance \cite{Ble23,BS23,HW15,HJ13,Lan14,Le22,Lin14,LS20,RWH17,Yam24}
	and the references therein.
	
	Le~\cite[Remark~1.2 and Theorem~1.3]{Le22} established Hadamard-type
	inequalities for $k$-positive matrices.
	
	\begin{thm}[Le]\label{thm:Le-kpositive}
		Let $1\le k\le n$ and let $A\in \sn$ be $k$-positive. Then
		\[
		S_k(\lambda(A)) \ \le\ S_k(\diag(A)).
		\]
		Moreover, equality holds if and only if $A$ is diagonal.
	\end{thm}
	
	\subsection{Hyperbolicity across PDE, combinatorics, and optimization}
	
	Hyperbolic polynomials were born in the analytic study of partial differential equations:
	in his foundational work~\cite{Gar51}, G{\aa}rding isolated hyperbolicity as the algebraic signature
	behind well-posedness of the Cauchy problem for linear PDE with constant coefficients.
	At a conceptual level, the associated hyperbolicity cones encode a preferred notion of
	``positivity'' compatible with real-rootedness and convexity; see also the modern expositions
	\cite{BGLS01,HL13}.
	
	In nonlinear analysis and geometry, these cones reappear as the natural ellipticity domains for
	fully nonlinear equations such as the $k$-Hessian operators, where the G{\aa}rding cones $\Gamma_k(n)$
	govern admissible Hessians and a priori estimates; cf.\ the program in~\cite{Le21,Le22}.
	In combinatorics and probability, real stable (equivalently, homogeneous hyperbolic) polynomials
	encode strong forms of negative dependence and interlacing, and have become a central tool in
	spectral graph theory and random processes; see, e.g., \cite{Ble23}.
	In optimization, hyperbolicity cones give rise to hyperbolic programming, providing a far-reaching
	generalization of semidefinite programming; see \cite{BGLS01,HL13} for geometric and convex-analytic viewpoints.
	
	This paper focuses on an interaction between \emph{majorization} and
	\emph{concavity on hyperbolicity cones}: the concavity of the order-th power of a hyperbolic polynomial, combined with permutation invariance, yields a Schur--Horn type principle on hyperbolicity cones.

	\subsection{Hyperbolic polynomials and hyperbolicity cones}\label{subsec:hyperbolic}
	
	We briefly recall the notion of hyperbolicity; see G{\aa}rding~\cite{Gar59}
	and the survey of Harvey--Lawson~\cite{HL13} for details.
	
	\begin{defn}[$a$-hyperbolic polynomial]\label{def:a-hyperbolic}
		Let $P$ be a homogeneous real polynomial of degree $k$ on $\R^n$.
		Given a direction $a\in\R^n$, we say that $P$ is \emph{hyperbolic with respect to $a$}
		(or \emph{$a$-hyperbolic}) if $P(a)>0$ and, for every $x\in\R^n$, the univariate polynomial
		$t\mapsto P(ta+x)$ has only real zeros.
	\end{defn}
	
	Equivalently, for each $x\in\R^n$ one may factor
	\[
	P(ta+x)\ =\ P(a)\prod_{i=1}^k \bigl(t+\lambda_i(P;a,x)\bigr)
	\qquad (t\in\R),
	\]
	where the real numbers $\lambda_i(P;a,x)$ are referred to as the \emph{$a$-eigenvalues} of $x$.
	
	\begin{defn}[Hyperbolicity cone]\label{def:garding-cone}
		Let $P$ be $a$-hyperbolic of degree $k$.
		The associated hyperbolicity (G{\aa}rding) cone is
		\[
		\Gamma(P;a)\ :=\ \bigl\{x\in\R^n:\ \lambda_i(P;a,x)>0\ \text{for all }i=1,\dots,k\bigr\}.
		\]
	\end{defn}
	
	G{\aa}rding's fundamental theorem implies that $\Gamma(P;a)$ is a nonempty open convex cone and, moreover,
	that the choice of hyperbolicity direction is irrelevant \emph{within a fixed cone}:
	if $b\in\Gamma(P;a)$, then $P$ is also hyperbolic with respect to $b$ and $\Gamma(P;b)=\Gamma(P;a)$;
	see \cite[Theorem~2]{Gar59} and \cite[Theorem~2.5 and Fact~2.7]{BGLS01}.
	
	Throughout the paper we fix the direction $a=\one$ and write
	\[
	\Gamma(P)\ :=\ \Gamma(P;\one)
	\]
	for the corresponding hyperbolicity cone.
	
	\subsection{Le's conjecture for symmetric hyperbolic polynomials}
	
	The following conjecture was formulated by Le in the context of hyperbolic polynomials;
	see \cite[Conjecture~5.5]{Le22}.
	
	\begin{conj}[Le]\label{conj:Le}
		Let $P$ be a homogeneous, real, symmetric polynomial of degree $k$ on $\R^n$ that is
		hyperbolic with respect to $\one$.
		Let $A\in\sn$ and suppose that $\lambda(A)\in \Gamma(P)$.
		Then $\diag(A)\in \Gamma(P)$ and
		\[
		P(\diag(A)) \ \ge\ P(\lambda(A)).
		\]
	\end{conj}

	Here ``symmetric'' means invariance under permutations of coordinates; see
	Definition~\ref{def:symmetric-poly} below. The conjecture may be viewed as a hyperbolic analogue of Hadamard inequality (Theorem \ref{thm:hadamard}).
	
	A basic example is obtained by taking $P=S_k$.
	Equivalently, for $A\in\sn$ one may define $P_k(A)$ implicitly by
	\begin{equation}\label{eq:Pk}
		\det(tI_n+A)\ =\ \sum_{k=0}^n t^{\,n-k} P_k(A)\qquad (t\in\R),
	\end{equation}
	so that $P_k(A)=S_k(\lambda(A))$.
	Moreover,
	\begin{equation}\label{eq:Pk-cone}
		\Gamma(P_k)\ =\ \bigl\{A\in \sn:\ \lambda(A)\in \Gamma_k(n)\bigr\};
	\end{equation}
	see \cite[Example~5.2]{Le22} and \cite[Eq.~(2.10)]{Le21}.
	Since $\Gamma(P_k)$ is convex, \eqref{eq:Pk-cone} immediately implies that the set of
	$k$-positive matrices is convex.
	
	For these $P_k$, Conjecture~\ref{conj:Le} recovers Le's Hadamard type inequalities
	for $k$-positive matrices.
	Our goal is to prove Conjecture~\ref{conj:Le} in full generality, for \emph{arbitrary}
	symmetric hyperbolic $P$, by combining G{\aa}rding's concavity with classical majorization theory.
	\begin{remark}
		For general (non-normal) matrices there is no Schur--Horn majorization relation
		between the diagonal and the spectrum, so the symmetric (or Hermitian) hypothesis in Conjecture \ref{conj:Le} is essential
		for the argument below.
	\end{remark}
	\subsection{Main result}\label{subsec:prior-art}
	
	From the viewpoint of majorization theory, the logical core of our main inequality is classical:
	a symmetric concave function is Schur-concave (Hardy--Littlewood--P\'olya \cite{HLP52},
	Marshall--Olkin--Arnold \cite{MOA11}; see also \cite{Bha97,Zha11}).
	If one already knows that $P^{1/k}$ is concave and permutation-invariant on a convex set,
	then Jensen's inequality and Birkhoff's theorem immediately imply Schur-concavity.
	
	The purpose of this paper is to record that, in the hyperbolic setting, these hypotheses are
	naturally satisfied on the hyperbolicity cone, and to package the resulting monotonicity as a
	\emph{Schur--Horn type principle on hyperbolicity cones}.  Concretely:
	
	\begin{itemize}[leftmargin=2.2em]
		\item We isolate the \emph{majorization closure} of $\Gamma(P)$: if $y\in\Gamma(P)$ and $x\maj y$,
		then automatically $x\in\Gamma(P)$.
		This makes Schur-concavity usable in ``cone-valued'' problems (and is exactly what is needed for Conjecture \ref{conj:Le}). 
		
		\item We show that if $P$ is symmetric and hyperbolic with respect to $\one$, then $P^{1/k}$
		is concave on $\Gamma(P)$ (G{\aa}rding) and Schur-concave with respect to the majorization order on $\Gamma(P)$.  Conjecture \ref{conj:Le} follows directly from the Schur--Horn theorem once the hyperbolic Schur-concavity principle is formulated in this cone-theoretic form. 
		
		\item Finally, we formulate a companion \emph{symmetrization principle} for general finite group actions on a
		hyperbolicity cone (Reynolds averaging), and we apply it to sign-flip symmetries of
		linear principal minor polynomials \cite{Ble23}.  This viewpoint explains why Fischer--Hadamard type pinching
		inequalities are, at heart, a concavity-and-symmetry phenomenon.
	\end{itemize}
	
	\subsection*{Organization}
	
In 	Section~\ref{sec:hyperbolic-Schur}, we state our main structural result is a hyperbolic analogue of Schur-concavity,
	stated as Theorem~\ref{thm:schurconcave}.
	As a direct corollary we obtain Theorem~\ref{thm:main}, which settles
	Conjecture~\ref{conj:Le}.
	In Section~\ref{sec:further}, we show how the same principle acts as a ``result generator'':
	any classical eigenvalue majorization relation immediately lifts to an inequality for
	\emph{arbitrary} symmetric hyperbolic polynomials (Corollary \ref{cor:Hermitian-part}--\ref{cor:block-diag}).
	We also formulate a general group-averaging monotonicity principle on hyperbolicity cones (Theorem \ref{thm:hyperbolic-symmetrization}) and use it to
	rederive Fischer--Hadamard inequalities for PSD-stable linear principal minor polynomials (Theorem \ref{thm:lpm-fischer}).
	
	% ============================================================
	\section{A Schur--Horn principle on hyperbolicity cones}\label{sec:hyperbolic-Schur}
	
	In this section we develop the hyperbolic Schur--concavity principle and deduce
	Conjecture \ref{conj:Le} from it.
	We begin by recalling a few standard notions from majorization theory, see
	\cite[Chapters~1--3]{MOA11} and also \cite[Chapter~2]{Bha97} or \cite[Chapter~10]{Zha11}.
	
	\subsection{Majorization and doubly stochastic matrices}
	
	\begin{defn}[Majorization]\label{def:majorization}
		Given $x,y\in\R^n$, let $x^{\downarrow}$ (resp.\ $y^{\downarrow}$) denote the vector obtained by
		rearranging the entries of $x$ (resp.\ $y$) in nonincreasing order.
		We say that $x$ is \emph{majorized} by $y$, written $x\maj y$, if
		\[
		\sum_{i=1}^k x_i^{\downarrow} \le \sum_{i=1}^k y_i^{\downarrow},\qquad k=1,\dots,n-1,
		\]
		and equality holds also for $k=n$.
	\end{defn}

	\begin{defn}[Doubly stochastic matrix]
		A \emph{doubly stochastic matrix} is a square nonnegative matrix whose row sums
		and column sums are all equal to $1$.
	\end{defn}
	
	The following characterization is classical; see e.g., \cite[Theorem~II.1.10]{Bha97} or \cite[Theorem~10.8]{Zha11}.
	
	\begin{lem}\label{lem:majorization_doubly}
		Let $x,y\in\R^n$. Then $x\maj y$ if and only if $x=Dy$ for some doubly stochastic matrix $D$.
	\end{lem}
	
	We also recall classic Schur--Horn's and Birkhoff's theorems; see e.g.,\cite[p.~22, Exercise II.1.12]{Bha97} and \cite[p.~159, Theorem 5.21]{Zha11}.
	
	\begin{lem}[Schur--Horn]\label{lem:Schur}
		Let $A\in \sn$. Then
		\[
		\diag(A)\ \maj\ \lambda(A).
		\]
		Equivalently, there exists a doubly stochastic matrix $S$ such that
		$\diag(A)=S\,\lambda(A)$.
	\end{lem}
	
	\begin{lem}[Birkhoff]\label{lem:Birkhoff}
		Every doubly stochastic matrix is a convex combination of permutation matrices.
	\end{lem}
	
	We now record a convenient form of G{\aa}rding's inequality; see  Güler \cite[Theorem~3.1 and Lemma~3.1]{Gul97}.
	
	\begin{lem}[G{\aa}rding concavity]\label{lem:garding_cone_concavity}
		Let $P$ be a homogeneous polynomial of degree $k$ that is hyperbolic with respect to $\one$,
		and let $\Gamma(P)=\Gamma(P;\one)$ be the (open) hyperbolicity cone. Then:
		\begin{enumerate}[label=\textup{(\roman*)}]
			\item $\Gamma(P)$ is a nonempty open convex cone.
			\item The map $x \mapsto P(x)^{1/k}$ is concave on $\Gamma(P)$.
		\end{enumerate}
	\end{lem}
	
	\subsection{Symmetry and permutation invariance}
	Let $\Sym$ denote the symmetric group on $[n]=\{1,\dots,n\}$, i.e., the set of all permutations of $[n]$.
	
	\begin{defn}[Symmetric polynomial]\label{def:symmetric-poly}
		A polynomial $P:\R^n\to\R$ is \emph{symmetric} if
		$P(\sigma x)=P(x)$ for all permutations $\sigma\in\Sym$ and all $x\in \R^n$.
	\end{defn}
	
	\begin{thm}[Permutation invariance]\label{thm:cone-perm}
		If $P$ is symmetric and hyperbolic with respect to $\one$, then $\Gamma(P)$ is invariant under coordinate permutations:
		if $x\in \Gamma(P)$ and $\sigma\in\Sym$, then $\sigma x\in \Gamma(P)$.
	\end{thm}
	
	\begin{proof}
		Fix $\sigma\in\Sym$ and $x\in \R^n$.
		Since $\sigma\one=\one$ and $P$ is symmetric, we have
		\[
		P(\sigma x+t\one)=P(\sigma x+t\sigma\one)=P(\sigma(x+t\one))=P(x+t\one).
		\]
		Hence the univariate polynomials $t\mapsto P(x+t\one)$ and $t\mapsto P(\sigma x+t\one)$ coincide, so their root sets agree.
		Equivalently, the multisets of $\one$-eigenvalues of $x$ and $\sigma x$ agree.
		Therefore $x\in\Gamma(P)\iff \sigma x\in\Gamma(P)$.
	\end{proof}
	
	\subsection{Schur-concavity on $\Gamma(P)$}
	Now, we introduce the definition of Schur-concave functions, see e.g., \cite[p.~80, Definition~A.1]{MOA11}. 
	
	\begin{defn}[Schur-concave function]\label{def:Schur-concave}
		A real-valued function $\phi$ defined on a subset of $\R^n$ is called
		\emph{Schur-concave} if
		\[
		x \maj y \quad \Longrightarrow \quad \phi(x)\ge \phi(y).
		\]
	\end{defn}
	
	\begin{thm}[Hyperbolic Schur-concavity]\label{thm:schurconcave}
		Let $P:\R^n\to\R$ be symmetric and hyperbolic with respect to $\one$, of degree $k$.
		Then:
		\begin{enumerate}[label=\textup{(\roman*)}]
			\item \emph{Majorization closure}: if $y\in\Gamma(P)$ and $x\maj y$, then $x\in \Gamma(P)$.
			\item \emph{Schur-concavity}: the map $f(x):=P(x)^{1/k}$ is Schur-concave on $\Gamma(P)$, i.e.,
			if $x,y\in\Gamma(P)$ and $x\maj y$, then
			\[
			P(x)^{1/k}\ \ge\ P(y)^{1/k}.
			\]
			In particular, $P(x)\ge P(y)$.
		\end{enumerate}
	\end{thm}
	
	\begin{proof}
		Assume $x\maj y$.
		By Lemma~\ref{lem:majorization_doubly}, there exists a doubly stochastic matrix $D$
		such that $x=Dy$.  By Lemma~\ref{lem:Birkhoff}, we can write
		\[
		D=\sum_{m=1}^N \alpha_m \Pi_m
		\]
		with $\alpha_m\ge0$, $\sum_{m=1}^N \alpha_m=1$, and permutation matrices $\Pi_m$.
		Thus
		\[
		x=\sum_{m=1}^N \alpha_m \Pi_m y.
		\]
		
		\smallskip\noindent\emph{(i) Majorization closure.}
		If $y\in\Gamma(P)$, then by Theorem~\ref{thm:cone-perm}
		each $\Pi_m y$ also belongs to $\Gamma(P)$.
		By the convexity of $\Gamma(P)$ from Lemma~\ref{lem:garding_cone_concavity}(i),
		we conclude that $x\in\Gamma(P)$.
		
		\smallskip\noindent\emph{(ii) Schur-concavity.}
		Define $f(x):=P(x)^{1/k}$ on $\Gamma(P)$.  By
		Lemma~\ref{lem:garding_cone_concavity}(ii), $f$ is concave, and by symmetry of $P$
		we have $f(\Pi_m y)=f(y)$ for all $m$.  Hence
		\[
		f(x)=f\!\left(\sum_{m=1}^N \alpha_m \Pi_m y\right)
		\ \ge\ \sum_{m=1}^N \alpha_m f(\Pi_m y)
		=\sum_{m=1}^N \alpha_m f(y)
		=f(y).
		\]
		Finally, note that $P>0$ on $\Gamma(P)$ (indeed, for $z\in\Gamma(P)$ one has
		$P(z)=P(\one)\prod_{i=1}^k \lambda_i(P;\one,z)$ with all factors $>0$), so raising both sides
		to the $k$th power yields $P(x)\ge P(y)$.
	\end{proof}
	
	\begin{remark}[Closed-cone version]\label{rem:closed-cone}
		Since $\Gamma(P)$ is an open cone and $P$ is continuous, the conclusions above extend to the closure.
		More precisely, if $y\in\overline{\Gamma(P)}$ and $x\maj y$, then $x\in\overline{\Gamma(P)}$ and $P(x)\ge P(y)$.
		
		Indeed, fix $\varepsilon>0$. Using
		\[
		P\bigl(t\one+(y+\varepsilon\one)\bigr)=P\bigl((t+\varepsilon)\one+y\bigr),
		\]
		the $\one$-eigenvalues satisfy $\lambda_i(P;\one,y+\varepsilon\one)=\lambda_i(P;\one,y)+\varepsilon$.
		Hence $y+\varepsilon\one\in\Gamma(P)$ whenever $y\in\overline{\Gamma(P)}$.
		
		Moreover, majorization is translation-invariant: if $x\maj y$, then $x+\varepsilon\one\maj y+\varepsilon\one$
		(since rearrangements are unaffected and each partial sum increases by $k\varepsilon$).
		Therefore Theorem~\ref{thm:schurconcave} applies to $x+\varepsilon\one$ and $y+\varepsilon\one$; letting
		$\varepsilon\downarrow0$ gives the claim.
	\end{remark}
	
	\subsection{Diagonal versus spectrum: proof of Conjecture \ref{conj:Le}}
	
	\begin{thm}\label{thm:main}
		Let $P:\R^n\to\R$ be symmetric and hyperbolic with respect to $\one$, of degree $k$.
		Let $A\in \sn$ satisfy $\lambda(A)\in \Gamma(P)$.
		Then:
		\begin{enumerate}[label=\textup{(\roman*)}]
			\item $\diag(A)\in \Gamma(P)$;
			\item $P(\diag(A))\ge P(\lambda(A))$.
		\end{enumerate}
	\end{thm}
	
	\begin{proof}
		By Schur--Horn's theorem (Lemma~\ref{lem:Schur}), we have $\diag(A)\maj \lambda(A)$.
		Since $\lambda(A)\in\Gamma(P)$, assertion~(i) follows from
		Theorem~\ref{thm:schurconcave}(i), and then
		Theorem~\ref{thm:schurconcave}(ii) yields $P(\diag(A))\ge P(\lambda(A))$.
	\end{proof}
			Thus Conjecture~\ref{conj:Le} holds.
	Taking $P=P_k$ as defined in~\eqref{eq:Pk} yields Le's Hadamard type inequality
	for $k$-positive matrices, recovering Theorem~\ref{thm:Le-kpositive}.
	Taking $P=S_n$ recovers Hadamard's determinant inequality (Theorem~\ref{thm:hadamard}).
	\begin{remark}[Boundary cases]\label{rem:boundary-main}
		By Remark~\ref{rem:closed-cone}, Theorem~\ref{thm:main} remains valid if one only assumes
		$\lambda(A)\in\overline{\Gamma(P)}$.
		This is often the natural formulation in matrix applications, where ``positivity'' cones are typically taken closed
		(e.g.,\ the positive semidefinite cone for Hadamard/Fischer inequalities).
	\end{remark}

	\subsection{A min--max principle for diagonals of orthogonal conjugates}
	
	Fix $A\in \sn$ with $\lambda(A)\in\Gamma(P)$ and consider the set of diagonals
	\[
	\mathcal{D}(A):=\{\diag(UAU^\top): U\in O(n)\}.
	\]
	The Schur--Horn theorem identifies $\mathcal{D}(A)$ with the permutahedron
	$\mathrm{conv}\{\sigma\lambda(A):\sigma\in\Sym\}$; see, e.g.,
	\cite[Theorem~10.8]{Zha11}.
	
	\begin{cor}[Diagonal extremizers on the permutahedron]\label{cor:minmax-diag}
		Let $P$ be symmetric hyperbolic with respect to $\one$, and let $A\in\sn$
		satisfy $\lambda(A)\in\Gamma(P)$.
		Then for every $U\in O(n)$,
		\[
		P(\lambda(A))
		\ \le\ P(\diag(UAU^\top))
		\ \le\ P\!\left(\frac{\tr A}{n}\,\one\right).
		\]
		Moreover, the left inequality is attained when $U$ diagonalizes $A$.
	\end{cor}
	
	\begin{proof}
		The left inequality is Theorem~\ref{thm:main} applied to $UAU^\top$,
		whose spectrum equals $\lambda(A)$.
		
		For the right inequality, note that among vectors with a fixed sum $s$,
		the constant vector $(s/n)\one$ is majorized by every other such vector.
		In particular,
		\[
		\frac{\tr A}{n}\one\ \maj\ \diag(UAU^\top).
		\]
		By Theorem~\ref{thm:main}(i) we have $\diag(UAU^\top)\in\Gamma(P)$, and then
		Theorem~\ref{thm:schurconcave}(i) (majorization closure) implies
		$\frac{\tr A}{n}\one\in\Gamma(P)$.
		Now Theorem~\ref{thm:schurconcave}(ii) yields
		\[
		P\!\left(\frac{\tr A}{n}\,\one\right)\ \ge\ P(\diag(UAU^\top)).
		\]
	\end{proof}
	
	\subsection{Rigidity under strict concavity}\label{subsec:rigidity}
	
	The Schur-concavity inequality is often strict, but strictness depends on the geometry of $\Gamma(P)$ and on whether
	$P^{1/k}$ is strictly concave.  We record a simple rigidity criterion.
	
	\begin{prop}[Strictness assuming strict concavity]\label{prop:strictness}
		Let $P$ be symmetric hyperbolic with respect to $\one$ of degree $k$, and assume that
		$f:=P^{1/k}$ is \emph{strictly} concave on $\Gamma(P)$.
		If $x,y\in\Gamma(P)$ satisfy $x\maj y$ and $x$ is not a permutation of $y$, then
		\[
		P(x)\ >\ P(y).
		\]
		In particular, if $A\in\sn$ satisfies $\lambda(A)\in\Gamma(P)$ and $\diag(A)$ is not a permutation of $\lambda(A)$,
		then $P(\diag(A))>P(\lambda(A))$.
	\end{prop}
	
	\begin{proof}
		As in the proof of Theorem~\ref{thm:schurconcave}, write $x=\sum_{m=1}^N \alpha_m \Pi_m y$
		with $\alpha_m>0$, $\sum_m\alpha_m=1$ and permutations $\Pi_m$.
		If $x$ is not a permutation of $y$, then at least two points $\Pi_m y$ appearing with positive weight are distinct,
		so $x$ is a nontrivial convex combination of distinct points in $\Gamma(P)$.
		Strict concavity gives
		\[
		f(x)=f\!\left(\sum_{m=1}^N \alpha_m \Pi_m y\right)\ >\ \sum_{m=1}^N \alpha_m f(\Pi_m y)=f(y),
		\]
		and raising to the $k$th power yields $P(x)>P(y)$.
		
		For the matrix statement, use $\diag(A)\maj\lambda(A)$ (Lemma~\ref{lem:Schur}).
	\end{proof}
	
	\begin{remark}
		For $P=S_n$, strict concavity of $\det^{1/n}$ on the positive definite cone implies the familiar strict
		Hadamard inequality unless $A$ is diagonal.
		For general hyperbolic polynomials, strict concavity is an additional hypothesis and is related to
		nondegeneracy properties of $P$; see \cite{HL13} for discussion of strict hyperbolicity and related convexity phenomena.
	\end{remark}
	
	% ============================================================
	\section{Applications and further remarks}\label{sec:further}
	
	This section collects two kinds of consequences.
	First, Theorem~\ref{thm:schurconcave} acts as a ``majorization-to-hyperbolicity'' transfer principle:
	any eigenvalue majorization relation immediately yields an inequality for \emph{arbitrary} symmetric hyperbolic polynomials.
	Second, we explain a complementary \emph{symmetrization} mechanism on hyperbolicity cones,
	which applies beyond the eigenvalue setting; in particular, it yields a concise proof of the
	hyperbolic Fischer--Hadamard inequalities for PSD-stable linear principal minor polynomials
	developed in~\cite{Ble23}.
	
	\subsection{Lifting eigenvalue majorization to hyperbolic inequalities}\label{subsec:template}
	
	Throughout this subsection, we fix a symmetric hyperbolic polynomial
	$P:\R^m\to\R$ (with respect to $\one$) of degree $k$.
	Whenever we encounter a majorization relation $x\maj y$ between two explicitly constructed vectors,
	Theorem~\ref{thm:schurconcave} immediately yields the inequality $P(x)\ge P(y)$ provided $y\in\Gamma(P)$.
	
	\begin{prop}[Majorization transfer principle]\label{prop:transfer}
		Let $x,y\in\R^m$ satisfy $x\maj y$.
		If $y\in\Gamma(P)$, then $x\in\Gamma(P)$ and $P(x)\ge P(y)$.
		Moreover, the same holds with $\Gamma(P)$ replaced by $\overline{\Gamma(P)}$.
	\end{prop}
	
	\begin{proof}
		This is exactly Theorem~\ref{thm:schurconcave} and Remark~\ref{rem:closed-cone}, specialized to dimension $m$.
	\end{proof}
	
	We now list several standard eigenvalue majorization relations from matrix analysis and indicate
	how Proposition~\ref{prop:transfer} turns them into hyperbolic inequalities.
	(Our goal is not exhaustiveness, but to illustrate the breadth of the ``one-line'' transfer mechanism.)
	
	\subsection{Fan-type and Schur-type inequalities}
	
	\begin{cor}[Hermitian part]\label{cor:Hermitian-part}
		Let $A$ be an $n\times n$ complex matrix with eigenvalues
		$\lambda(A)=(\lambda_1(A),\dots,\lambda_n(A))$ and let
		$\Re\lambda(A)=(\Re\lambda_1(A),\dots,\Re\lambda_n(A))$.
		Let $\mathcal{H}(A)=(A+A^*)/2$ denote the Hermitian part of $A$.
		If $\lambda(\mathcal{H}(A))\in\Gamma(P)$, then
		\[
		\Re\lambda(A)\in\Gamma(P)
		\quad\text{and}\quad
		P(\Re\lambda(A))\ \ge\ P(\lambda(\mathcal{H}(A))).
		\]
	\end{cor}
	
	\begin{proof}
		Fan's theorem~\cite[Theorem~10.28]{HJ13} asserts that
		\[
		\Re\lambda(A)\ \maj\ \lambda(\mathcal{H}(A)).
		\]
		Apply Proposition~\ref{prop:transfer}.
	\end{proof}
	
	\begin{cor}[Real parts of diagonals for normal matrices]\label{cor:normal-Re}
		Let $A=(a_{ij})$ be a $n\times n$ normal matrix, and set
		$d(A)=(a_{11},\dots,a_{nn})$.
		If $\Re\lambda(A)\in\Gamma(P)$, then
		\[
		\Re d(A)\in\Gamma(P)
		\quad\text{and}\quad
		P(\Re d(A))\ \ge\ P(\Re\lambda(A)).
		\]
	\end{cor}
	
	\begin{proof}
		Exercise~7 in \cite[Section~10.4]{HJ13} states that
		\[
		\Re d(A)\ \maj\ \Re\lambda(A).
		\]
		Apply Proposition~\ref{prop:transfer}.
	\end{proof}
	
	\begin{cor}[Rank-one perturbations]\label{cor:rank-one}
		Let $A\in\sn$ and let $x$ be any complex $n$-row vector.
		Set $B=A+x^*x$.  If $\lambda(B)\in\Gamma(P)$, then
		\[
		P(\lambda(B))\ \ge\
		P\bigl(\lambda_1(A)+\|x\|^2,\lambda_2(A),\dots,\lambda_n(A)\bigr).
		\]
	\end{cor}
	
	\begin{proof}
		Exercise~3 in \cite[Section~10.4]{HJ13} shows that
		\[
		\lambda(B)\ \maj\ \bigl(\lambda_1(A)+\|x\|^2,\lambda_2(A),\dots,\lambda_n(A)\bigr).
		\]
		Apply Proposition~\ref{prop:transfer}.
	\end{proof}
	
	\begin{cor}[Sums of positive semidefinite matrices]\label{cor:sum-PSD}
		Let $A,B$ be $n\times n$ positive semidefinite matrices.
		Consider the $2n$-vectors
		\[
		u=(\lambda(A+B),0,\dots,0),\qquad v=(\lambda(A),\lambda(B)).
		\]
		Let $P:\R^{2n}\to\R$ be symmetric hyperbolic with respect to $\one$.
		If $v\in\Gamma(P)$, then $u\in\Gamma(P)$ and $P(u)\ge P(v)$.
	\end{cor}
	
	\begin{proof}
		Exercise~14 in \cite[Section~10.4]{HJ13} states that
		\[
		(\lambda(A+B),0)\ \maj\ (\lambda(A),\lambda(B)).
		\]
		Apply Proposition~\ref{prop:transfer}.
	\end{proof}
	
	\begin{cor}[Eigenvalues of a sum and a difference]\label{cor:Fan-sum-diff}
		Let $A,B\in\sn$.
		\begin{enumerate}[label=\textup{(\roman*)}]
			\item If $\lambda(A+B)\in\Gamma(P)$, then
			\[
			P(\lambda(A+B))\ \ge\ P\bigl(\lambda(A)+\lambda(B)\bigr),
			\]
			where the sum is taken componentwise.
			\item If $\lambda(A-B)\in\Gamma(P)$, then
			\[
			P\bigl(\lambda(A)-\lambda(B)\bigr)\ \ge\ P(\lambda(A-B)).
			\]
		\end{enumerate}
	\end{cor}
	
	\begin{proof}
		Fan's inequalities \cite[Theorems~10.21 and 10.22]{HJ13} assert that
		\[
		\lambda(A+B)\ \maj\ \lambda(A)+\lambda(B),
		\qquad
		\lambda(A)-\lambda(B)\ \maj\ \lambda(A-B).
		\]
		Apply Proposition~\ref{prop:transfer}.
	\end{proof}
	
	\subsection{Pinching and block operations}
	
	\begin{cor}[Principal block-diagonal submatrices]\label{cor:block-diag}
		Let $A$ be an $n\times n$ Hermitian matrix partitioned as
		\[
		A=\begin{pmatrix}A_{11} & A_{12} \\ A_{21} & A_{22}\end{pmatrix}.
		\]
		Let $B=A_{11}\oplus A_{22}$.
		If $\lambda(A)\in\Gamma(P)$, then
		\[
		(\lambda(A_{11}),\lambda(A_{22}))\in\Gamma(P)
		\quad\text{and}\quad
		P(\lambda(A_{11}),\lambda(A_{22}))\ \ge\ P(\lambda(A)).
		\]
	\end{cor}
	
	\begin{proof}
		Exercise~2 in \cite[Section~10.4]{HJ13} shows that
		$\lambda(B)=(\lambda(A_{11}),\lambda(A_{22}))$ and that
		$\lambda(B)\maj\lambda(A)$.
		Apply Proposition~\ref{prop:transfer}.
	\end{proof}
	
	\begin{remark}[A matrix-level symmetrization viewpoint]
		Many pinching operations (block-diagonal projection, diagonal projection, conditional expectations onto
		$*$-subalgebras) can be expressed as convex combinations of unitary conjugations.
		This makes them instances of the group-averaging principle developed in Subsection~\ref{subsec:symmetrization} below.
	\end{remark}
	
	\subsection{A symmetrization principle on hyperbolicity cones}\label{subsec:symmetrization}
	
	The proofs above are instances of a single geometric phenomenon:
	\emph{convexification by symmetries} improves the value of a concave invariant functional.
	We now record a formulation that will later be applied to linear principal minor polynomials.
	
	\begin{defn}[Reynolds operator, \cite{DK15}]\label{def:Reynolds}
		Let $V$ be a real vector space and let $G$ be a finite group acting linearly on $V$.
		The \emph{Reynolds operator} (or group average) is the linear map
		\[
		\Rey_G:V\to V,\qquad
		\Rey_G(x)\ :=\ \frac{1}{|G|}\sum_{g\in G} g\cdot x.
		\]
	\end{defn}
	
	\begin{thm}[Hyperbolic symmetrization principle]\label{thm:hyperbolic-symmetrization}
		Let $V$ be a real vector space and let $\mathcal{P}:V\to\R$ be a homogeneous polynomial of degree $k$
		which is hyperbolic with respect to some $e\in V$, with hyperbolicity cone $\Gamma(\mathcal{P};e)$.
		Assume that a finite group $G$ acts linearly on $V$ such that
		\begin{enumerate}[label=\textup{(\roman*)}]
			\item $g\cdot e=e$ for all $g\in G$;
			\item $\mathcal{P}(g\cdot x)=\mathcal{P}(x)$ for all $g\in G$ and all $x\in V$.
		\end{enumerate}
		Then, for every $x\in\Gamma(\mathcal{P};e)$,
		\begin{enumerate}[label=\textup{(\alph*)}]
			\item $g\cdot x\in\Gamma(\mathcal{P};e)$ for all $g\in G$, and $\Rey_G(x)\in\Gamma(\mathcal{P};e)$;
			\item $\mathcal{P}(\Rey_G(x))\ge \mathcal{P}(x)$.
			Equivalently, the concave function $\mathcal{P}^{1/k}$ is monotone under symmetrization by $G$ on $\Gamma(\mathcal{P};e)$.
		\end{enumerate}
	\end{thm}
	
	\begin{proof}
		Fix $g\in G$ and $x\in V$. Since $g\cdot e=e$ and $\mathcal{P}$ is $G$-invariant, we have
		\[
		\mathcal{P}(g\cdot x+t e)=\mathcal{P}(g\cdot x+t\, g\cdot e)=\mathcal{P}\bigl(g\cdot (x+t e)\bigr)=\mathcal{P}(x+t e).
		\]
		Thus the univariate polynomials $t\mapsto \mathcal{P}(x+t e)$ and $t\mapsto \mathcal{P}(g\cdot x+t e)$ have the same roots.
		In particular, $x\in\Gamma(\mathcal{P};e)$ if and only if $g\cdot x\in\Gamma(\mathcal{P};e)$.
		This proves $G$-invariance of $\Gamma(\mathcal{P};e)$.
		
		Since $\Gamma(\mathcal{P};e)$ is convex and open (by G{\aa}rding's theory; cf.\ \cite{Gar59,BGLS01}),
		and each $g\cdot x$ lies in $\Gamma(\mathcal{P};e)$, the average $\Rey_G(x)$ also lies in $\Gamma(\mathcal{P};e)$.
		This proves~(a).
		
		For~(b), set $f:=\mathcal{P}^{1/k}$ on $\Gamma(\mathcal{P};e)$.
		By G{\aa}rding concavity, $f$ is concave on $\Gamma(\mathcal{P};e)$, and by $G$-invariance of $\mathcal{P}$,
		we have $f(g\cdot x)=f(x)$ for all $g\in G$. Hence Jensen's inequality gives
		\[
		f(\Rey_G(x))
		=f\!\left(\frac{1}{|G|}\sum_{g\in G} g\cdot x\right)
		\ \ge\ \frac{1}{|G|}\sum_{g\in G} f(g\cdot x)
		=\frac{1}{|G|}\sum_{g\in G} f(x)
		=f(x).
		\]
		Raising both sides to the $k$th power yields $\mathcal{P}(\Rey_G(x))\ge \mathcal{P}(x)$.
	\end{proof}
	
	\begin{remark}\label{rem:unifying}
		Theorem~\ref{thm:hyperbolic-symmetrization} can be viewed as a group-theoretic shadow of
		Theorem~\ref{thm:schurconcave}.  In the eigenvalue setting, the role of $G$ is played by the full symmetric group,
		and ``symmetrization'' corresponds to convex combinations of permutations (doubly stochastic matrices).
		In the matrix-minor setting of \cite{Ble23}, the relevant symmetries are not permutations of coordinates but sign-flip
		conjugations, leading to Fischer--Hadamard type \emph{pinching} maps.
	\end{remark}
	
	\subsection{Pinching inequalities for linear principal minor polynomials}\label{subsec:lpm-pinching}
	
	We now explain how the symmetrization principle above recovers, in a concise and conceptual form,
	a key inequality from the work of Blekherman et al.~\cite{Ble23} on
	\emph{linear principal minor polynomials} (lpm polynomials).
	
	\subsubsection*{Linear principal minor polynomials and sign symmetries}
	
	Let $X$ be an $n\times n$ symmetric matrix of indeterminates.
	For $J\subseteq[n]$, let $X_J$ denote the principal submatrix indexed by $J$.
	An \emph{lpm polynomial} is a finite linear combination
	\[
	\mathcal{P}(X)\ :=\ \sum_{J\subseteq[n]} c_J \det(X_J).
	\]
	When $\mathcal{P}$ is PSD-stable (equivalently, hyperbolic with respect to the identity matrix),
	its hyperbolicity cone $H(\mathcal{P})$ provides a natural replacement for the positive semidefinite cone
	and supports a rich family of determinantal inequalities; see \cite{Ble23}.
	
	Let $\mathcal{D}_n:=\{\mathrm{diag}(\varepsilon_1,\dots,\varepsilon_n):\varepsilon_i\in\{\pm1\}\}$ be the group of
	diagonal sign matrices.  It acts on $\sn$ by conjugation, $A\mapsto DAD$.
	
	\begin{lem}[Sign-flip invariance of lpm polynomials]\label{lem:lpm-sign}
		Every lpm polynomial $\mathcal{P}$ is invariant under diagonal sign-flip conjugation:
		\[
		\mathcal{P}(DAD)=\mathcal{P}(A)\qquad\text{for all }A\in\sn\text{ and all }D\in\mathcal{D}_n.
		\]
		Consequently, if $\mathcal{P}$ is hyperbolic with respect to $I_n$, then $H(\mathcal{P})$ is invariant under $A\mapsto DAD$.
	\end{lem}
	
	\begin{proof}
		It suffices to check invariance for each principal minor.
		Fix $J\subseteq[n]$.  Writing $D_J$ for the restriction of $D$ to the coordinates in $J$, we have
		\[
		(DAD)_J = D_J\, A_J\, D_J.
		\]
		Hence
		\[
		\det((DAD)_J)=\det(D_J A_J D_J)=\det(D_J)^2\,\det(A_J)=\det(A_J),
		\]
		since $\det(D_J)\in\{\pm1\}$.  By linearity in the coefficients $c_J$, this yields $\mathcal{P}(DAD)=\mathcal{P}(A)$.
		
		For the cone invariance, note that $D I_n D = I_n$. Thus for every $t\in\R$,
		\[
		\mathcal{P}(DAD+t I_n)=\mathcal{P}\bigl(D(A+t I_n)D\bigr)=\mathcal{P}(A+t I_n),
		\]
		so the univariate root sets agree and membership in $H(\mathcal{P})$ is preserved by $A\mapsto DAD$.
	\end{proof}
	
	\subsubsection*{Block pinching as a Reynolds operator}
	
	Let $\Pi=\{S_1,\dots,S_m\}$ be a partition of $[n]$.
	Write $\pi_\Pi:\sn\to\sn$ for the \emph{block-diagonal pinching} (orthogonal projection) which
	sets $A_{ij}=0$ whenever $i$ and $j$ lie in different blocks of $\Pi$.
	Equivalently, $\pi_\Pi(A)$ is the block-diagonal matrix with diagonal blocks $A_{S_\ell}$.
	
	Define the subgroup
	\[
	\mathcal{D}_\Pi
	:=\Bigl\{\mathrm{diag}\bigl(\varepsilon_1\mathbf{1}_{S_1},\dots,\varepsilon_m\mathbf{1}_{S_m}\bigr):\varepsilon_\ell\in\{\pm1\}\Bigr\}
	\subseteq \mathcal{D}_n,
	\]
	whose elements flip signs \emph{blockwise}.
	
	\begin{lem}[Pinching is group averaging]\label{lem:pinching-average}
		For every $A\in\sn$,
		\[
		\pi_\Pi(A)=\Rey_{\mathcal{D}_\Pi}(A)=\frac{1}{|\mathcal{D}_\Pi|}\sum_{D\in\mathcal{D}_\Pi} DAD.
		\]
	\end{lem}
	
	\begin{proof}
		Let $A=(a_{ij})$.  For a fixed pair $(i,j)$, the $(i,j)$-entry of $DAD$ equals $d_i d_j a_{ij}$.
		If $i$ and $j$ belong to the same block of $\Pi$, then $d_i=d_j$ for every $D\in\mathcal{D}_\Pi$, hence $d_i d_j=1$
		and the average preserves $a_{ij}$.  If $i$ and $j$ are in different blocks, then $d_i d_j$ takes values $\pm1$
		equally often as $D$ ranges over $\mathcal{D}_\Pi$, and the average is $0$.
		Thus the average matrix agrees entrywise with $\pi_\Pi(A)$.
	\end{proof}
	
	\subsubsection*{A hyperbolic Fischer--Hadamard inequality via concavity and symmetries}
	
	We can now recover the hyperbolic Fischer--Hadamard inequality for PSD-stable lpm polynomials
	from \cite[Theorem~2.4]{Ble23}.
	
	\begin{thm}[Hyperbolic Fischer--Hadamard inequality for lpm polynomials]\label{thm:lpm-fischer}
		Let $\mathcal{P}$ be a homogeneous PSD-stable lpm polynomial of degree $k$ on $\sn$
		(hence hyperbolic with respect to $I_n$), and let $\Pi$ be a partition of $[n]$.
		Then for every $A\in H(\mathcal{P})$,
		\[
		\pi_\Pi(A)\in H(\mathcal{P})
		\qquad\text{and}\qquad
		\mathcal{P}(\pi_\Pi(A))\ \ge\ \mathcal{P}(A).
		\]
	\end{thm}
	
	\begin{proof}
		By Lemma~\ref{lem:lpm-sign}, $\mathcal{P}$ and its hyperbolicity cone $H(\mathcal{P})$ are invariant under
		$A\mapsto DAD$ for all $D\in\mathcal{D}_\Pi$.
		In particular, $DAD\in H(\mathcal{P})$ whenever $A\in H(\mathcal{P})$.
		Since $H(\mathcal{P})$ is convex and $\pi_\Pi(A)$ is the average of the orbit (Lemma~\ref{lem:pinching-average}),
		we get $\pi_\Pi(A)\in H(\mathcal{P})$.
		
		Now apply Theorem~\ref{thm:hyperbolic-symmetrization} to the vector space $V=\sn$,
		the polynomial $\mathcal{P}$ (hyperbolic with respect to $I_n$), and the group $G=\mathcal{D}_\Pi$ acting by conjugation.
		The conclusion yields
		\[
		\mathcal{P}(\pi_\Pi(A))=\mathcal{P}(\Rey_{\mathcal{D}_\Pi}(A))\ \ge\ \mathcal{P}(A).
		\]
	\end{proof}
	
	\begin{cor}[Hyperbolic Hadamard inequality for lpm polynomials]\label{cor:lpm-hadamard}
		In Theorem~\ref{thm:lpm-fischer}, take $\Pi=\{\{1\},\dots,\{n\}\}$.
		Then $\pi_\Pi(A)=\diag(A)$ (diagonal pinching), and
		\[
		\mathcal{P}(\diag(A))\ \ge\ \mathcal{P}(A)
		\qquad\text{for all }A\in H(\mathcal{P}).
		\]
	\end{cor}
	
	\begin{remark}
		The argument above isolates a minimal mechanism:
		\emph{(i) invariance under a symmetry group, (ii) convexity of the hyperbolicity cone, (iii) G{\aa}rding concavity}.
		In \cite{Ble23}, Fischer--Hadamard type inequalities are proved in a broader stability framework and are complemented by
		stronger determinantal inequalities (including hyperbolic Koteljanskii-type inequalities) obtained via stability-preserver
		and negative dependence technology.
		Our proof does not replace those deeper tools; rather, it clarifies that the Fischer--Hadamard direction is, at heart,
		a concavity-and-symmetry phenomenon.
	\end{remark}
	
%	\subsection*{Open problems}
%	
%	We conclude by recording a couple of concrete questions suggested by the above discussion.
%	
%	\begin{prob}[Equality and rigidity]
%		For a given symmetric hyperbolic polynomial $P$, describe when equality in Theorem~\ref{thm:main},
%		\[
%		P(\diag(A))=P(\lambda(A)),
%		\]
%		can occur.  Under what additional hypotheses on $P$ does equality force $A$ to be diagonal up to orthogonal conjugacy?
%	\end{prob}
%	
%	\begin{prob}[Beyond full symmetry]
%		To what extent can one formulate a basis-invariant diagonal-versus-spectrum comparison
%		for hyperbolic polynomials without full invariance under $\Sym$, perhaps by replacing majorization
%		with a partial-order coming from a smaller group (orbit convexity) as in Theorem~\ref{thm:hyperbolic-symmetrization}?
%	\end{prob}
	
	% ============================================================
	\section*{Acknowledgments}
	
	Teng Zhang is supported by the China Scholarship Council, the Young Elite Scientists
	Sponsorship Program for PhD Students (China Association for Science and Technology),
	and the Fundamental Research Funds for the Central Universities at Xi'an Jiaotong
	University (Grant No.~xzy022024045).
	
	% ============================================================


\begin{thebibliography}{99}
		
		\bibitem{BGLS01}
		H. H. Bauschke, O. G{\"u}ler, A. S. Lewis and H. S. Sendov,
		Hyperbolic polynomials and convex analysis,
		Canad. J. Math. {\bf 53} (3) (2001), 470--488.
		
		\bibitem{Bha97}
		R.~Bhatia,
		\emph{Matrix Analysis},
		Graduate Texts in Mathematics, Vol.~169,
		Springer-Verlag, New York, 1997.
		
		\bibitem{Ble23}
		G. Blekherman, M. Kummer, R. Sanyal, K. Shu and S. Sun,
		Linear principal minor polynomials: hyperbolic determinantal inequalities
		and spectral containment,
		Int. Math. Res. Not. IMRN {\bf 2023}, no.~24, 21346--21380.
		
		\bibitem{BS23}
		P. Braun and H.~S. Sendov,
		On the Hadamard-Fischer inequality, the inclusion-exclusion formula,
		and bipartite graphs,
		Linear Algebra Appl. {\bf 668} (2023), 64--92.
		
		\bibitem{DK15}
		H. Derksen and G. Kemper,
		\emph{Computational Invariant Theory}, 2nd ed.,
		Encyclopaedia of Mathematical Sciences, Vol.~130,
		Springer, Heidelberg, 2015.
		
		\bibitem{Gar51}
		L. G{\aa}rding,
		Linear hyperbolic partial differential equations with constant coefficients,
		Acta Math. {\bf 85} (1951), 1--62.
		
		\bibitem{Gar59}
		L.~G{\aa}rding,
		An inequality for hyperbolic polynomials,
		J. Math. Mech. \textbf{8} (1959), no.~6, 957--965.
		
		\bibitem{Gul97}
		O. G{\"u}ler,
		Hyperbolic polynomials and interior point methods for convex programming,
		Math. Oper. Res. \textbf{22} (1997), no.~2, 350--377.
		
		\bibitem{HL13}
		F.~R.~Harvey and H.~B.~Lawson, Jr.,
		G{\aa}rding's theory of hyperbolic polynomials,
		Comm. Pure Appl. Math. \textbf{66} (2013), no.~7, 1102--1128.
		
		\bibitem{HLP52}
		G. H. Hardy, J. E. Littlewood and G. P\'olya,
		\emph{Inequalities}, 2nd ed.,
		Cambridge University Press, Cambridge, 1952.
		
		\bibitem{HW15}
		C.~J. Hillar and A. Wibisono,
		A Hadamard-type lower bound for symmetric diagonally dominant positive matrices,
		Linear Algebra Appl. {\bf 472} (2015), 135--141.
		
		\bibitem{HJ13}
		R.~A.~Horn and C.~R.~Johnson,
		\emph{Matrix Analysis}, 2nd ed.,
		Cambridge University Press, Cambridge, 2013.
		
		\bibitem{Lan14}
		K.~L. Lange,
		Hadamard's determinant inequality,
		Amer. Math. Monthly {\bf 121} (2014), no.~3, 258--259.
		
		\bibitem{Le21}
		N.~Q.~Le,
		A spectral characterization and an approximation scheme for the Hessian eigenvalue,
		Rev. Mat. Iberoam. \textbf{38} (2022), no.~5, 1473--1500.
		
		\bibitem{Le22}
		N.~Q. Le,
		Hadamard-type inequalities for $k$-positive matrices,
		Linear Algebra Appl. {\bf 635} (2022), 159--170.
		
		\bibitem{Lin14}
		M. Lin,
		An Oppenheim type inequality for a block Hadamard product,
		Linear Algebra Appl. {\bf 452} (2014), 1--6.
		
		\bibitem{LS20}
		M. Lin and G.~J. Sinnamon,
		Revisiting a sharpened version of Hadamard's determinant inequality,
		Linear Algebra Appl. {\bf 606} (2020), 192--200.
		
		\bibitem{MOA11}
		A. W. Marshall, I. Olkin and B. C. Arnold,
		\emph{Inequalities: Theory of Majorization and Its Applications}, 2nd ed.,
		Springer Series in Statistics,
		Springer, New York, 2011.
		
		\bibitem{RWH17}
		M. R\'o\.za\'nski, R. Witu\l a{} and E. Hetmaniok,
		More subtle versions of the Hadamard inequality,
		Linear Algebra Appl. {\bf 532} (2017), 500--511.
		
		\bibitem{Yam24}
		A. Yamada,
		Oppenheim--Schur's inequality and RKHS,
		Oper. Matrices {\bf 18} (2024), no.~3, 711--733.
		
		\bibitem{Zha11}
		F. Zhang,
		\emph{Matrix Theory: Basic Results and Techniques}, 2nd ed.,
		Universitext, Springer, New York, 2011.
		
	\end{thebibliography}
\end{document}